\documentclass[12pt,a4paper]{amsart}
\usepackage[utf8]{inputenc}

\usepackage{graphicx,color}
\usepackage{amsmath}
\usepackage{amssymb}
\usepackage{amsfonts}
\usepackage{amsthm}
\usepackage{amscd}
\usepackage{ae}
\usepackage[T1]{fontenc}

\newcommand{\D}{\mathbb{D}}

\font\sets=msbm10 scaled \magstep1
\def\R{\text{\sets R}}

\def\N{\text{\sets N}}

\def\C{\text{\sets C}}

\newcommand{\fff}{\mathcal{F}}

\renewcommand\Re{\operatorname{Re}}
\renewcommand\Im{\operatorname{Im}}
\renewcommand\arg{\operatorname{arg}}

\newcommand{\stex}{\mathcal{X}}

\newcommand{\CBB}{\textcolor{blue}}

\theoremstyle{plain}
\newtheorem{theorem}{Theorem}
\newtheorem{proposition}[theorem]{Proposition}
\newtheorem{lemma}[theorem]{Lemma}
\newtheorem{corollary}[theorem]{Corollary}

\author[O. Hirviniemi]{Olli Hirviniemi}
\address{University of Helsinki, Department of Mathematics and Statistics, P.O. Box 68 , FIN-00014 University of Helsinki, Finland}
\email{olli.hirviniemi@helsinki.fi}

\author[I. Prause]{Istv\'an Prause}
\address{Department of Physics and Mathematics, University of Eastern Finland, P.O. Box 111, 80101 Joensuu, Finland}
\email{istvan.prause@uef.fi}

\author[E. Saksman]{Eero Saksman}
\address{University of Helsinki, Department of Mathematics and Statistics, P.O. Box 68 , FIN-00014 University of Helsinki, Finland}
\email{eero.saksman@helsinki.fi}

\thanks{The work was supported by the Finnish Academy Coe `Analysis and Dynamics' and the Finnish Academy projects  1266182, 1303765, and 1309940. }

\title[Quasiconformal stretching and rotation on a line]{Stretching and rotation of planar quasiconformal mappings on a line}

\begin{document}

\begin{abstract}
In this article, we examine stretching and rotation of planar quasiconformal mappings on a line. We show that for almost every point on the line, the set of complex stretching exponents (describing stretching and rotation jointly) is contained in  the disk $ \overline{B}(1/(1-k^4),k^2/(1-k^4))$. This yields a quadratic improvement over the known optimal estimate for general sets of Hausdorff dimension $1$. Our proof is based on holomorphic motions and estimates for dimensions of quasicircles. We also give a lower bound for the dimension of the image of a $1$-dimensional subset of a line under a quasiconformal mapping.
\end{abstract}

\maketitle

\smallskip



\noindent {\bf Data availability:} Not applicable

\section{Introduction}\label{se:intro}

Let $f: \C \to \C$ be a $K$-quasiconformal mapping. The classical local H\"older continuity result 
\[
C^{-1}|z-z_0|^{K} \leq |f(z)-f(z_0)| \leq C|z-z_0|^{\frac{1}{K}} 
\]
where $C=C(f,z_0,K)$, known as Mori's theorem (see e.g. \cite[Theorem 3.10.2]{AIM}) describes the extremal stretching properties of $f$ at every point $z_0 \in \C$. Bounds for local rotation are obtained from  \cite[Theorem 3.1]{AIPS}:
\begin{equation}\label{eq:gamma}
\limsup_{z \to z_0} \frac{|\arg (f(z)-f(z_0))|}{|\log|f(z)-f(z_0)||} \leq \frac{1}{2}\left( K - \frac{1}{K} \right).
\end{equation}
Moreover, as argument is obtained from the imaginary part of the complex logarithm, they can be studied jointly by considering the behaviour of the fraction
\[
\frac{\log(f(z)-f(z_0))}{\log(z-z_0)}
\]
as $z \to z_0$. The set of accumulations points  $\alpha(1+i\gamma)$ of the fraction is called $\stex_f(z_0)$ -- we refer to the precise definition \eqref{eq:cse} in next section. The real part  $\alpha>0$ measures the local stretching, and the ratio of the imaginary and the real part $\gamma\in\R$ the local rotation at $z_0$.

The multifractal spectra estimates from \cite{AIPS} give the optimal upper bound for the Hausdorff dimension of the set where prescribed stretching and rotation can happen. Namely, 
if for every $z \in E\subset \C$ 
we have $\alpha(1+i\gamma)\in \stex_f(z)$, 
then by \cite[Theorem 1.4]{AIPS}
\[
\dim E  \leq 1 + \alpha-\frac{1}{k}\sqrt{(1-\alpha)^2+(1-k^2)\alpha^2\gamma^2},
\]
and this upper bound is sharp. Here $k = (K-1)/(K+1)$. 

In order to give for the stated bound  a more accessible geometric interpretation, we note that equivalently for any $0 \leq s \leq 2$, for almost all points $z$ with respect  to the $s$-dimensional Hausdorff measure the set of complex stretching exponents $\stex_f(z)$
lie in the closed disk with the geometric diameter
\begin{equation}\label{eq:geometric}
\left[ \frac{1-k}{1+k} + \frac{k}{1+k}s , \frac{1+k}{1-k} - \frac{k}{1-k}s \right].
\end{equation}

In this paper we study  behaviour of quasiconformal maps on  lines, in which situation  it  is natural to expect more constrained behaviour than on general 1-dimensional sets. Indeed, this is quantified by our main result as follows:
\begin{theorem} \label{th:stex1}
Let $\phi : \C \to \C$ be a $K$-quasiconformal mapping. For almost every $x \in \R$ with respect to 1-dimensional Lebesgue measure we have that $\mathcal{X}_{\phi}(x) \subset \overline{B}(1/(1-k^4),k^2/(1-k^4))$. Here $k=(K-1)/(K+1)$.
\end{theorem}

\noindent One should note that the bound \eqref{eq:geometric} in case $s=1$ only yields that  $\stex_f(x)$ is included in a disk  with the diameter $[1/(1+k), 1/(1-k)]$, while Theorem \ref{th:stex1} gives as geometric diameter the segment $[1/(1+k^2), 1/(1-k^2)]$. In terms of pure rotation, the following corollary follows immediately from Theorem \ref{th:stex1}.

\begin{corollary}\label{co:rotation}
Let $\phi : \C \to \C$ be a $K$-quasiconformal mapping. For almost every $x \in \R$ with respect to 1-dimensional Lebesgue measure we have that 
\[
\gamma = \limsup_{r\to 0} \Big| \frac{\arg(\phi(x+r)-\phi(x))}{\log|\phi(x+r)-\phi(x)|} \Big| \leq \frac{k^2}{\sqrt{1-k^4}},
\]
where $k=(K-1)/(K+1)$.
\end{corollary}

We also obtain the following estimate  for the Hausdorff dimension of the images of subsets of line. The sharp estimate for the dimension of the image of a general $1$-dimensional set $A$ was given in \cite{Astala} as $\dim f(A) \geq 1- k$.

\begin{theorem}\label{cor:ldim}
For any $K$-quasiconformal mapping $f \colon \C \to \C$ and $A \subset \R$ with the Hausdorff dimension 1, $\dim f(A) \geq 1 - k^2$ for $k = (K-1)/(K+1)$.
\end{theorem}
\noindent  This estimate generalizes results from \cite{Pr07} and \cite{PrSm}  where it was assumed that  $f(\R)\subset\R$, and gives a natural counterpart of the estimate $\dim f(\R) \leq 1 + k^2$ from \cite{Smi}. In a similar manner, Theorem \ref{th:stex1} can be viewed as a rotational counterpart of \cite{Smi}. Indeed, the proofs of Theorems \ref{th:stex1} and \ref{cor:ldim} follow the general line of argument of \cite{Astala,PrSm} based on holomorphic motions and pressure estimates adapted to our situation.

\smallskip

\noindent \textbf{Acknowledgements}.\; We are grateful for the referee for many useful comments which improved the readability of the text.

\section{Prerequisites}\label{se:prereq}

We say that 
an orientation preserving homeomorphism $f: \Omega \to \Omega'$ is \emph{$K$-quasiconformal} with $K \geq 1$ if $f \in W^{1,2}_{loc}(\Omega)$ and for almost every $z \in \Omega$ the directional derivatives satisfy
\[
\max_{\alpha} |\partial_{\alpha}f(z)| \leq K \min_{\alpha} |\partial_{\alpha}f(z)|.
\]

Equivalently, we could define quasiconformal mappings via the Beltrami equation: a homeomorphism $f: \Omega \to \Omega'$ between two planar domains is a $K$-quasiconformal mapping if it lies in the Sobolev space $W^{1,2}_{loc}(\Omega)$ and satisfies the Beltrami equation
\[
f_{\overline{z}} = \mu f_z
\]
for some measurable function $\mu : \Omega \to \C$ with $|\mu(z)|\leq k < 1$ where $k = (K-1)/(K+1)$. As $k<1$, we can say that $f$ is $k$-quasiconformal without ambiguity.

We also consider the notion of quasisymmetry. If $\eta : [0,\infty) \to [0, \infty)$ is an increasing homeomorphism, then $f : \Omega \to \Omega'$ is $\eta$-quasisymmetric if for any distinct points $z_0, z_1, z_2 \in \Omega$ we have
\[
\frac{|f(z_0)-f(z_1)|}{|f(z_0)-f(z_2)|} \leq \eta\left( \frac{|z_0-z_1|}{|z_0-z_2|} \right).
\]

Of particular interest is the case where $\Omega = \C$. In this case, any $K$-quasiconformal mapping is $\eta$-quasisymmetric, where $\eta$ depends only on $K$.
The measurable Riemann mapping theorem implies that for any measurable function $\mu$ with $|\mu(z)|\leq k < 1$ there is a unique normalized solution $f$ solving the corresponding Beltrami equation and having $f(0)=0$ and $f(1)=1$.

For more information about quasiconformal mappings, see for example \cite{AIM}.

We  measure the local rotation of a quasiconformal mapping  by comparing the image of a line near a point to logarithmic spirals. 
If the limit
\[
\gamma:=\lim_{r \to 0+} \frac{\arg(f(z_0+r)-f(z_0))}{\log|f(z_0+r)-f(z_0)|}
\]
exists, the rate of rotation can be uniquely defined to be the value of this limit. In general, the limit might not exist, in which case any accumulation point of the fraction as $r \to 0+$ can be called a \emph{rate of rotation} at $z_0$.  One may note that  the different  choices of a continuous branch for the argument  on the image curve differ by a constant, and hence do  not influence the set of accumulation points.  We refer to \cite{AIPS} for a more thorough discussion and equivalent definitions defining the concept of rotation  for quasiconformal maps. 

In fact, it turns out to be useful to consider  local stretching and rotation jointly. Then one compares the local behaviour to that of  the complex power map $f(z) = f(z_0) + \omega \cdot \frac{z-z_0}{|z-z_0|}|z-z_0|^{\alpha(1+i\gamma)}$ defined in the disk $B(z_0,r)$, where $\omega \in \C \setminus \{0\}$, $\alpha > 0$ and $\gamma \in \R$. For these mappings, $|f(z)-f(z_0)| = |\omega| |z-z_0|^\alpha$ and $\arg(f(z)-f(z_0)) = \arg \omega + \arg(z-z_0) + \alpha \gamma \log|z-z_0|$. This means that $\alpha$ measures the stretching of the map near $z_0$, while $\gamma$ determines the local geometric rate of rotation.  


To define the joint exponents  in a  precise way, let $f : \Omega \to \Omega'$ be quasiconformal and $z_0 \in \Omega$ be a point. We introduce the set of \emph{complex stretching exponents} $\stex_f(z_0) \subset \C$ by setting 
\begin{equation}\label{eq:cse}
\stex_f(z_0) = \bigcap_{0 < t_0 \leq 1} \overline{\left\lbrace \frac{\log(f(z_0+t)-f(z_0))}{\log(t)} : 0 < t < t_0  \right\rbrace},
\end{equation}
or in other words, the limit points of the quotient $\log(f(z_0+t)-f(z_0))/\log(t)$ as $t \to 0$. This definition  does not  depend on the chosen continuous branch of the complex logarithm on $f((z_0, z_0+t_0])-f(z_0)$. Note that a branch can be defined along this curve by fixing the value at $f(z_0+t_0)-f(z_0)$. With this definition, we are only approaching $z_0$ along one ray, but again this does not affect the set of possible limits. The  advantage of complex stretching exponents is that they allow for  \emph{holomorphic} dependence which is particularly useful when embedded into a holomorphic motion.

Finally, we remark that if $\tau=\alpha(1+i\gamma)$ is a complex stretching exponent at $z_0$, then $\gamma=\Im \tau/\Re \tau$ is a rate of rotation at $z_0$. The converse holds as well: if $\gamma$ is a rate of rotation at $z_0$, then by Mori's theorem the real parts of corresponding fractions for stretching exponents are bounded, and by compactness we can find a complex stretching exponent $\alpha(1+i\gamma)$ at $z_0$.

\section{Proofs of Theorems \ref{th:stex1} and \ref{cor:ldim}}\label{se:proofs}


We begin by proving the following auxiliary lemma for holomorphic mappings.

\begin{lemma} \label{lem:holo}
Let $h: \D \to \D$ be a holomorphic function such that $h(0)= \varepsilon$ and for all $r\in (0,1)$ we have the inclusion $h(r\D) \subset B((1-r^2)/2,(1+r^2)/2)$. Then for  any  fixed $0 \leq k < 1$ we have
\[
|h(k)| \leq k^2 + o(1)
\]
as $\varepsilon \to 0$.
\end{lemma}

\begin{proof}
Consider the following family $ \fff_{\varepsilon}$ of holomorphic functions
\[
\fff_\varepsilon = \{f : \D \to \D : |f(z)-(1-|z|^2)/2| \leq (1+|z|^2)/2, |f(0)|\leq \varepsilon\}.
\]
Let $\psi(\varepsilon) := \sup \{|f(k)| : f \in \fff_{\varepsilon}\}$. As the family $\fff_{\varepsilon}$ is normal by Montel's theorem, for each $\varepsilon>0$ we may find a sequence $(f_n)$ in $\fff_{\varepsilon}$, converging uniformly on compact subsets, such that $|f_n(k)| \to \psi(\varepsilon)$. It is easy to see that the limit function $h_\varepsilon: =\lim_{n\to\infty} f_n$ belongs to $\fff_{\varepsilon}$.

Clearly $\psi$ is decreasing in $\varepsilon$ and we need   to show that $\lim_{\varepsilon\to 0}\psi(\varepsilon) \leq k^2$, or equivalently that $\lim_{m\to\infty}  |h_{1/m}(k)|\leq k^2.$  Using the fact that $\fff_1$ is normal, we can further assume that $(h_{1/m})$ converges uniformly on compact subsets to a limit function $g$. Especially, then $\lim_{m\to\infty}  h_{1/m}(k)=g(k).$

The limit function $g:\D\to \D$ is analytic and satisfies $g(0)=0$ together with
\begin{equation}\label{eq:SchDisk}
|g(z)-(1-|z|^2)/2| \leq (1+|z|^2)/2
\end{equation}
for any $z \in \D$. If $g'(0)\neq 0$, then consider small values of $z$ in the direction $-\overline{g'(0)}$. The series expansion of $g$ is $g(z)=g'(0)z+O(|z|^2)$, which yields values with real part smaller than $-|z|^2$ for $z$ sufficiently close to $0$. This together with \eqref{eq:SchDisk} implies that $g'(0)=0$.

Finally,  the Schwarz lemma yields that $|g(z)/z| \leq |z|$. In particular, $|g(k)| \leq k^2$, which proves the lemma.
\end{proof}

The following proposition contains crucial ingredients of the proof and  after it  both Theorem \ref{th:stex1} and Theorem \ref{cor:ldim} follow rather effortlessly. As in similar consideration before (see e.g. \cite{Astala}, \cite{Smi}, \cite{PrSm}) the idea is to embed our quasiconformal map into a holomorphic motion. Assuming that the exponents deviate from the value $1$ in a large subset of $\R$ will then -- together with their holomorphic dependence -- lead to effective estimates via application of the connection to the thermodynamical pressure  characterisation of the Hausdorff dimension. A key new ingredient here is the use the fact that our set lies on the real axis, which is utilized in  Lemma \ref{lem:qcircle} and by the dimension estimate of quasicircles \cite{BP,Smi}, see especially equation \eqref{eq:apu}.

\begin{proposition} \label{th:discEstimate} 
Let $0\leq k<1$ and let $\phi: \C \to \C$ be a $k$-quasiconformal mapping with $\phi(0)=0$ and $\phi(1)=1$. Assume that $k < \rho < 1$ and $0 < \delta \leq 1$. There exists a constant $a = a(\rho) > 0$ with the following property: for any finite or countably infinite collection of disjoint disks $B(x_j,r_j)$ contained in $\D$ with $x_j \in \R$ and $\sum_j(ar_j)^\delta \geq 1$, there exists a probability distribution $p = (p_j)$ such that

\begin{equation}\label{eq:techni1}
\Re\left( \frac{-\sum_{j}p_j \log p_j}{-\sum_j p_j \log \Big(a\big(\phi(x_j+r_j)-\phi(x_j)\big)\Big)} \right) \geq 1 - (k/\rho)^2 - R(1-\delta)
\end{equation}
and
\begin{equation}\label{eq:techni2}
\frac{\sum_j p_j \log \Big(a\big(\phi(x_j+r_j)-\phi(x_j)\big)\Big)}{\sum_j p_j \log (ar_j)} \in \bigcup_{b\in [\delta,1]} \overline{B}(b/(1-s^2),bs/(1-s^2)), 
\end{equation}
where $s = (k/\rho)^2 + R(1-\delta)$ for some function $R$ with $R(t) \to 0$ as $t \to 0$. 
\end{proposition}

\begin{proof} Let us first assume that the given collection of disks is finite.
Let $\mu$ be the Beltrami coefficient of $\phi$, and set
\[
\mu_\lambda = \lambda \frac{\mu}{k}.
\]
This defines a family of Beltrami coefficients that depend analytically on $\lambda\in\D$. If we denote by $\phi_\lambda$ the unique solutions of the corresponding Beltrami equations with fixed points $0$ and $1$, then $\phi_0$ is the identity mapping of $\C$. Then $\phi_\lambda$ also depends analytically on $\lambda$ by Ahlfors-Bers theorem \cite{AhBe}. Moreover, the original $\phi$ can obtained from $\phi_k = \phi$.

 Recall that  $\rho\in(k,1)$. The mappings $\phi_\lambda$ have Beltrami coefficients bounded uniformly above by $\rho$ for all $|\lambda|< \rho$. Thus by uniform quasiconformality, the mappings $\phi_\lambda$ are  uniformly $\eta$-quasisymmetric for some $\eta = \eta_\rho$, see \cite[Theorem 3.5.3]{AIM}. This implies weak quasisymmetry: whenever $|x-z| \leq |y-z|$ and $|\lambda| < \rho$, we have
\[
\frac{|\phi_\lambda(x)-\phi_\lambda(z)|}{|\phi_\lambda(y)-\phi_\lambda(z)|} \leq \eta\left( \frac{|x-z|}{|y-z|}\right) \leq \eta(1) = C =C(\rho).
\]
Then the disks $B(\phi_\lambda(x_j),\frac{1}{C}|\phi_\lambda(x_j+r_j)-\phi_\lambda(x_j)|)$ are disjoint because they are contained in disjoint sets $\phi_\lambda (B(x_j,r_j))$. Moreover, these disks are contained in $B(0,C)$ as $|\phi_\lambda(z)-\phi_\lambda(0)|/|\phi_\lambda(1)-\phi_\lambda(0)| \leq \eta(|z|) \leq C$ for $z\in\D$, and hence we have the following holomorphic family of disks in unit disk:
\begin{equation}\label{eq:disks} 
B\left(\frac{1}{C}\phi_\lambda(x_j), \frac{1}{C^2}|\phi_\lambda(x_j+r_j)-\phi_\lambda(x_j)| \right).
\end{equation}

We fix the constant in the theorem by setting $a= 1/C^2$. Let $C_\lambda$ be the Cantor set generated by the contractive similarities of the unit disk onto the disks in the above collection \eqref{eq:disks} with rotation in directions given by $\phi_\lambda(x_j+r_j)-\phi_\lambda(x_j)$. Then $C_\lambda$ lies on the image of the real line under a $|\lambda|/\rho$-quasiconformal mapping. We refer to Lemma \ref{lem:qcircle} below for this fact and a more detailed definition of the Cantor set.

Let us denote the \emph{complex radii} by $r_j(\lambda) = a(\phi_\lambda(x_j+r_j)-\phi_\lambda(x_j)) \in \mathbb{C}\setminus{\{0\}}$. We apply Jensen's inequality to the pressure function to obtain the following auxiliary result for any strictly positive probability distribution $p$ (meaning that $p_j>0$, $\sum_j p_j= 1$) and $d\in(0,2]$:
\begin{equation} \label{eq:pressure}
P_\lambda(d) := \log \sum_j |r_j(\lambda)|^d = \log \sum_j p_j \frac{|r_j(\lambda)|^d}{p_j} \geq \sum_j p_j\log \frac{|r_j(\lambda)|^d}{p_j} = I_p - d \Re \Lambda_p(\lambda),
\end{equation}
with the equality reached when all $|r_j(\lambda)|^d / p_j$ have the same value. This application of Jensen's inequality is from \cite{Astala} and one may note that it is special instance of the so-called \emph{variational principle} in thermodynamics. Above, $I_p$ is the \emph{entropy}
\[
I_p = -\sum_j p_j \log p_j
\]
and $\Lambda_p$ is the \emph{complex Lyapunov exponent}
\[
\Lambda_p(\lambda) = - \sum_j p_j \log r_j(\lambda).
\]
By choosing a holomorphic branch of the logarithm, the function $\Lambda_p(\lambda)$ becomes holomorphic in $\lambda$. By Remark 2.3 of \cite{AIPS} this choice of the branch is consistent with the geometric definition in Section 2.

Let $d_\lambda$ be the Hausdorff dimension of $C_\lambda$. It easily follows that $P_\lambda(d_\lambda)=0$. In fact, when this observation  is combined with the basic dimension formula for self-similar fractals  (see e.g. \cite[Theorem 4.14]{Mat}). we see that  the following are equivalent:

\medskip

(i) \,\, $d \leq d_\lambda$.

\smallskip

(ii) \, $P_\lambda(d) \geq 0$.

\smallskip

(iii) There is a probability distribution $p$ such that $I_p - d \Re \Lambda_p(\lambda) \geq 0$.

\medskip

By assumption of the Proposition we have $\sum (ar_j)^\delta \geq 1$, or in other words, $P_0(\delta) \geq 0$.
Let $p$ be the maximizer of $I_p - \delta \Re \Lambda_p(0)$ in \eqref{eq:pressure} and define the holomorphic function
\[
\Phi(\lambda) = 1- \frac{I_p}{\Lambda_p(\lambda)}, \quad \lambda\in\rho\D
\]

We choose the branches of the logarithms in the sum for $\Lambda_p$ so that $\Im \log r_j(0) = 0$. As obviously $\dim C_\lambda \leq 2$, it follows that $I_p \leq 2 \Re \Lambda_p(\lambda)$ for any $|\lambda|<\rho$, and thus $\Phi(\rho\D) \subset \D$.

For any $\lambda$ with $|\lambda |<\rho$, Lemma \ref{lem:qcircle} below verifies that $C_\lambda$ lies on a $|\lambda|/\rho$-quasicircle. As any such quasicircle has a Hausdorff dimension at most $1+|\lambda|^2/\rho^2$ by \cite{Smi}\footnote{We remark here that we could use a weaker estimate for the dimension of quasicircles of the form $1+C k^2$ (for $C<\infty$) of \cite{BP} and modify the assumptions in Lemma \ref{lem:holo} accordingly -- its conclusion still remains valid.}, it follows that $I_p \leq (1+|\lambda|^2/\rho^2) \Re \Lambda_p(\lambda)$. This means that
\begin{equation}\label{eq:apu}
\Phi(\lambda) \in \overline{B}((1-|\lambda|^2/\rho^2)/2,(1+|\lambda|^2/\rho^2)/2),
\end{equation}
or in other words, $\Phi$ maps a disk of radius $r$ centered at origin into a disk with (geometric) diameter $[-r^2/\rho^2,1]$.

Since the logarithms in the sum defining $\Lambda_p(0)$ are real, $I_p-\delta \Lambda_p(0) \geq 0$, or $I_p/\Lambda_p(0) \geq \delta$. It follows that $0 \leq \Phi(0) \leq 1-\delta$.
Lemma \ref{lem:holo} applied to $\lambda \mapsto \Phi(\rho \lambda)$ then implies that
\begin{equation}\label{eq:holoEst}
|\Phi(k)| \leq (k/\rho)^2 + R(1-\delta), 
\end{equation}
where $R$ is a mapping from $[0,1]$ that has $R(t) \to 0$ as $t \to 0$. This proves the first statement (\ref{eq:techni1}) of the theorem.
We have
\[
\frac{\sum_j p_j \log  r_j(k)}{\sum_j p_j \log r_j(0)} = \frac{1-\Phi(0)}{1-\Phi(k)},
\]
where $\delta \leq 1-\Phi(0) \leq 1$, and $1-\Phi(k) \in B(1,(k/\rho)^2 + R(1-\delta))$. Setting $s = (k/\rho)^2 + R(1-\delta)$, this implies that
\[
\frac{\sum_j p_j \log  r_j(k)}{\sum_j p_j \log r_j(0)} \in  \overline{B}(b/(1-s^2),bs/(1-s^2))
\]
for some $b\in[\delta, 1] $. This proves the second statement (\ref{eq:techni2}) of the theorem. Since our estimates are uniform in the number of the disks, the case of infinitely many disks can be obtained from the finite case by a simple limiting argument. The proof is complete.
\end{proof}

During the proof of the preceding proposition, the Cantor set determined by the disks in the holomorphic motion lies always on a quasicircle if the disks are centered on the real line for $\lambda = 0$, see \cite{Pr07}. We provide the details of this for the convenience of the reader.

\begin{lemma} \label{lem:qcircle}
In the proof of Proposition \ref{th:discEstimate}, $C_\lambda$ lies on a $|\lambda|/\rho$-quasiconformal image of the real line.
\end{lemma}

\begin{proof}
The disks
\[
B\left(\frac{1}{C}\phi_\lambda(x_j), \frac{1}{C^2}|\phi_\lambda(x_j+r_j)-\phi_\lambda(x_j)| \right)
\]
have disjoint closures for every $|\lambda|<\rho$. Letting $w_j(\lambda)$ and $r_j(\lambda):=\frac{1}{C^2}(\phi_\lambda(z_j+r_j)-\phi_\lambda(z_j))$ be their centers and complex radii, we observe that both are holomorphic functions. All similarities of the form $\gamma_{j,\lambda}(z) = r_j(\lambda)z + w_i(\lambda)$ are strict contractions, and $C_\lambda$ is the unique non-empty compact set such that
\[
C_\lambda = \bigcup_{j} \gamma_{j,\lambda}(C_\lambda).
\]

Any iterated map $\gamma_{j_1,\lambda} \circ \gamma_{j_2,\lambda} \circ \ldots \circ \gamma_{j_k,\lambda}$  has a unique fixed point that belongs to the Cantor set $C_\lambda$ by the basic construction of $C_\lambda $ as a  self-similar fractal. The set of such fixed points is easily seen to be  dense in $C_\lambda$ since their closure $F$ is a non-empty compact set such that
\begin{equation}
\label{eq:Cantor}
F = \bigcup_{j} \gamma_{j,\lambda}(F).
\end{equation}
This equality follows from the fact that if $z_0$ is a fixed point of some iterated map $\gamma$ and $\gamma_0$ is any of the contractions, the sequence of fixed points of mappings $\gamma_0 \circ \gamma^k$ converges to $\gamma_0(z_0)$ as $k$ goes to infinity. By uniqueness of the Cantor set with property \eqref{eq:Cantor}, $F = C_\lambda$.

For any such fixed point $z = z(\lambda)$, we observe that $z(0) \mapsto z(\lambda)$ defines a holomorphic mapping. Using continuity, we can define a mapping $\Psi_\lambda :  C_0 \to C_\lambda$. We have found a holomorphic motion $\Psi: \rho \D \times C_0 \to \C$. Extended $\lambda$-lemma \cite{Slo91,AIM} allows us to extend this motion as $\Psi: \rho \D \times \C \to \C$, and as $C_0 \subset \R$, the claim follows. 
\end{proof}

Let $A$ be a subset of real line with Hausdorff dimension $1$. As a first application of Proposition \ref{th:discEstimate}, we obtain the lower bound for the dimension of the image of $A$ under $k$-quasiconformal mapping, i.e. our Theorem \ref{cor:ldim}.

\begin{proof}[Proof of Theorem \ref{cor:ldim}]
We may assume that $f(0)=0$ and $f(1)=1$. Let $0 < \delta < 1$ and $k < \rho < 1$ be arbitrary. Then as $A$ has Hausdorff dimension larger than $\delta$, some intersection $A \cap [n,n+1)$ must have infinite $\delta$-dimensional Hausdorff measure. Without loss of generality, we may therefore assume that $A \subset \D$.

Fix $\varepsilon > 0$. Let $\{B_\alpha\}_{\alpha \in \mathcal{A}}$ be a covering of $f(A)$ with disks of small enough diameters that $f^{-1}(B_{\alpha})$ have a diameter at most $\varepsilon$. Then these preimages are contained in disks of radius equal to the diameter of the original preimage set, let us call these sets $D_\alpha$. An application of Vitali's covering theorem allows us to take a disjoint countable subcollection $\{D_j = B(x_j,r_j)\}_{j \in \N}$ of these disks such that any of the disks $\{D_\alpha\}_{\alpha \in \mathcal{A}}$ is contained in some disk $5D_j = B(x_j,5r_j)$.

Let $a$ be the constant from Proposition \ref{th:discEstimate} for $k$, $\delta$ and $\rho$. By choosing $\varepsilon$ small enough, we have $\sum_{j \in \N} (10r_j)^{\delta} \geq 10^\delta a^{-\delta}$ since the set $A$ has infinite $\delta$-dimensional Hausdorff measure. Proposition \ref{th:discEstimate} can therefore be applied to this collection of disks, and estimate (\ref{eq:techni1}) yields that for a suitable probability distribution $(p_j)$ we have
\[
\Re\left( \frac{-\sum_{j}p_j \log p_j}{-\sum_j p_j \log \Big(a\big(f(x_j+r_j)-f(x_j)\big)\Big)} \right) \geq 1 - (k/\rho)^2 - R(1-\delta).
\]
It follows that
\[
-\sum_{j}p_j \log p_j - \big( 1 - (k/\rho)^2 - R(1-\delta) \big) \Big( -\sum_j p_j \log \big|a\big(f(x_j+r_j)-f(x_j)\big)\big| \Big) \geq 0,
\]
or equivalently
\[
 \sum_j p_j \log \Big(p_j^{-1}\big|a\big(f(x_j+r_j)-f(x_j)\big)\big|^{1 - (k/\rho)^2 - R(1-\delta)} \Big) \geq 0,
\]
and thus by Jensen's inequality 
\[
 \sum \big|a\big(f(x_j+r_j)-f(x_j)\big)\big|^{1 - (k/\rho)^2 - R(1-\delta)} \geq 1.
\]
We conclude that $\dim f(A) \geq 1-(k/\rho)^2 - R(1-\delta)$ as the original covering of $f(A)$ was arbitrary, only chosen to have small enough diameters. This holds for any $k < \rho < 1$, so letting $\rho \to 1$ shows that $\dim f(A) \geq 1-k^2 - R(1-\delta)$. Finally letting $\delta \to 1$ yields $\dim f(A) \geq 1-k^2$, finishing the proof.
\end{proof}

Next, we  prove our main result.

\begin{proof}[Proof of Theorem \ref{th:stex1}]
We can assume without loss of generality that $\phi(0)=0$ and $\phi(1)=1$. Let $E \subset \C$ be any open convex set with positive distance from $B := \overline{B}(1/(1-k^4),k^2/(1-k^4))$, and $A \subset \R$ be the set of those points $x$ with $\stex_\phi(x) \cap E \neq \emptyset$. It suffices to show that $A$ has measure $0$.

As $E$ and $B$ have positive distance, we can find $0 < \delta < 1$ and $k < \rho < 1$ such that
\[
F := \bigcup_{\delta \leq b \leq 1} b\overline{B}(1/(1-s^2),s/(1-s^2))
\]
has positive distance from $E$. Here $s = (k/\rho)^2 + R(1-\delta)$ and $R$ is from the conclusion of Proposition \ref{th:discEstimate}. Let $a$ be the constant in this theorem corresponding to this choice of $\rho$.

Fix $0 < \varepsilon < 1$. For any $x \in A$, we can find $0 < r_x < \varepsilon$ such that
\[
\frac{\log\big(\phi(x+r_x)-\phi(x)\big) + \log a}{\log(r_x)+\log a} \in E.
\]
Using Vitali covering theorem, we can find a countable collection of disjoint disks $B(x_j,r_{x_j})$ such that for any $y \in A$ there is $j$ such that $B(y,r_y) \subset B(x_j, 5r_{x_j})$. We observe that $m(A) \leq \sum_j 5r_{x_j}$.

On the other hand, if $\sum_{j} (ar_{x_j})^\delta \geq 1$, we can use Proposition \ref{th:discEstimate} to find a probability distribution $p$ such that
\[
\frac{ \sum_j p_j \log\Big(a\big(\phi(x_j+r_{x_j})-\phi(x_j)\big)\Big)}{\sum_j p_j \log(ar_{x_j})} \in F.
\]
But the left hand side is a convex combination of terms belonging in $E$, namely
\[
\frac{ \sum_j p_j \log\Big(a\big(\phi(x_j+r_{x_j})-\phi(x_j)\big)\Big)}{\sum_j p_j \log(ar_{x_j})} = \sum_j \left(\frac{-p_j \log (ar_{x_j})}{-\sum_{\ell}p_{\ell} \log (ar_{x_{\ell}})} \right)  \frac{\log\big(\phi(x_j+r_{x_j})-\phi(x_j)\big) + \log a}{\log(r_{x_j})+\log a}.
\]
This implies that $F \cap E \neq \emptyset$, a contradiction. Therefore, we must have $\sum_{j} (ar_{x_j})^\delta < 1$.

We have obtained
\[
m(A) \leq 5 \sum_j r_{x_j} \leq 5 \sum_j r_{x_j}^\delta \varepsilon^{1-\delta} \leq 5a^{-\delta} \varepsilon^{1-\delta}.
\]
As $0 < \varepsilon < 1$ was arbitrary, it follows that $m(A)=0$.

As the complement of $B$ is a countable union of half-planes with positive distance from $B$, the claim follows.

\end{proof}
 Finally,  Corollary \ref{co:rotation} is obtained easily from Theorem \ref{th:stex1}.
 
 \begin{proof}[Proof of Corollary \ref{co:rotation}] One simply observes that the  slope of any line from the origin to a point in the disk $\overline{B}(1/(1-k^4),k^2/(1-k^4))$ lies on the interval   $[-k^2(1-k^4)^{-1/2}, k^2(1-k^4)^{-1/2}]$.
 \end{proof}

\section{Further comments and discussion}\label{se:comments}

We first discuss the sharpness of our results. In our proof, we used Smirnov's  $(1 + k^2)$-estimate for the dimension of $k$-quasicircles \cite{Smi}, but it should be noted that this bound is not sharp. It was shown by Ivrii in \cite{Ivrii} that the asymptotical expansion of the upper bound for small $k$ is of the form $1 + \Sigma^2 k^2 + O(k^{2.5})$, with a constant $\Sigma^2$ which is strictly smaller than one by the work of Hedenmalm \cite{Hedenmalm}. This sharper version could be used to improve our results slightly since in Lemma \ref{lem:holo} we would have a stricter constraint for the holomorphic function.

The sharpness of the rotational results is tied to the sharpness in stretching. We can embed a quasiconformal mapping $f$ with (close to) extremal stretching behaviour, having $|f(x+t)-f(x)|$ comparable to $t^{1-ck^2}$ for $x$ in a large set, in a holomorphic motion, and use that to find a mapping $g$ with a rate of rotation $ck^2$ on the same set, for small values of $k$. For the readers benefit, let us describe with some more  detailed   (albeit still heuristic)  manner  how this is done. To that end,  let $\phi: \C \to \C$ be a $k$-quasiconformal mapping such that $\stex_\phi(x) \cap \{\Re z < 1-ck^2\} \neq \emptyset$ for $x \in A \subset \R$ with $m(A) > 0$, and let it be normalized as $\phi(0)=0$, $\phi(1)=1$. Letting $\mu$ be the associated Beltrami coefficient, embed this $\phi$ in a holomorphic motion as in our proof by setting $\mu_\lambda = \lambda \frac{\mu}{k}$ and let $\phi_\lambda$ solve the Beltrami equation with this coefficient.
For any $x \in A$ and $r>0$ the function $h_{x,r}: \D \to \C$ defined as $h_{x,r}(\lambda) = \log(\phi_\lambda(x+r)-\phi_\lambda(x))/\log r - 1$ is holomorphic with $h_{x,r}(0) = 0$. Without loss of generality we may assume that for every $x \in A$ the mappings $h_{x,r}$ have the same degree $d$ at origin. Set $\omega = e^{-\frac{\pi}{2d}i}$. Then for sufficiently small $r$ we have $\Re h_{x,r}(k) < -ck^2$ and hence $\Im h_{x,r}(\omega k) > ck^2 + O(k^{d+1})$. This basically implies that $\phi_{\omega k}$ has rotational properties of similar order as the original $\phi$ has stretching properties. For concrete examples with strong stretching on a large set, we refer to \cite[Theorem 5.1]{AIPP} where the considered Julia sets are images of the unit circle under $k$-quasiconformal mappings and their dimensions behave quadratically in $k$, which in turn implies the desired stretching behaviour.

We have formulated our results in terms of a distortion on a line for simplicity. Naturally, the same estimates hold on the unit circle. If one instead considers a conformal map in the unit disk then stronger results are available. Namely, according to a well-known theorem of Makarov \cite{Mak85} the stretching exponent $\alpha=1$ a.e. This has been extended by Binder \cite{Bind97} to cover rotation, accordingly, $\gamma=0$ a.e.  

It is interesting to contrast our results to examples of Julia sets. Due to a classical result of Ruelle \cite{Rue} the Julia set of $z^2+\lambda z$ has Hausdorff dimension $1 + \frac{1}{16\log 2}|\lambda|^2+ O(|\lambda|^3).$ In  \cite{Bind97} Binder established an analogous formula for the geometric rotation, which says that the rotation is  $\frac{\sin \arg \lambda}{64 \log 2} |\lambda|^3+ O(|\lambda|^4)$ almost everywhere with respect to the Hausdorff measure on the Julia set. This yields rotation of lower order than what we obtain for quasiconformal mappings in this paper (but with respect to a different measure). However, the dimension ends up being of the same order for both.

The above results concern typical, that is a.e.~behaviour. See \cite{Pr19} for the opposite end-point, i.e. for results effective for dimension close to zero. It would be interesting to extend these for the intermediate regime and ask what happens outside of a fixed dimension $s \in (0,1)$. One could, in principle, follow a similar approach to what we have presented here but in order to get effective estimates one would seem to need detailed multifractal estimates for harmonic measure.

The improved regularity obtained in Theorem \ref{th:stex1} compared to the estimate \cite{AIPS} naturally stems from the fact that we are considering subsets of the real axis. More generally, it would be interesting to see what other kind of structural assumptions on the underlying set could replace the real axis in this type of results. 
One possibility is to consider Jordan curves, or more generally continua,  with given dimension or given smoothness.

\end{document}